\documentclass[a4paper,10pt]{amsart}

\usepackage{hyperref}
\hypersetup{
    colorlinks=true,       % false: boxed links; true: colored links
    linkcolor=blue,          % color of internal links
    citecolor=blue,        % color of links to bibliography
    filecolor=blue,      % color of file links
    urlcolor=blue           % color of external links
}
\usepackage[all]{xy}
\usepackage{tikz,graphicx}
\usepackage{amssymb}
\usepackage{booktabs}

\theoremstyle{plain}
\newtheorem{theorem}{Theorem}[section]
\newtheorem{lemma}[theorem]{Lemma}

\newtheorem{corollary}[theorem]{Corollary}
\newtheorem{remark}[theorem]{Remark}

\theoremstyle{definition}
\newtheorem{defn}[theorem]{Definition}

\newcommand{\s}{\vspace{0.3cm}}

\def\Aut{\operatorname{Aut}}
\def\PGL{\operatorname{PGL}}

\def\Gal{\operatorname{Gal}}

\usepackage{enumerate}

\begin{document}

\title{Riemann surfaces defined over the reals}
\author[Eslam Badr] {Eslam Badr}
\address{$\bullet$\,\,Eslam Badr}
\address{Departament Matem\`atiques, Edif. C, Universitat Aut\`onoma de Barcelona\\
08193 Bellaterra, Catalonia, Spain} \email{eslam@mat.uab.cat}
\address{Department of Mathematics,
Faculty of Science, Cairo University, Giza-Egypt}
\email{eslam@sci.cu.edu.eg}
\thanks{Eslam Badr is supported by MTM2016-75980-P}

\author{Rub\'en A. Hidalgo}
\author{Sa\'ul Quispe}
\address{$\bullet$\,\,Rub\'en A. Hidalgo and Sa\'ul Quispe}
\address{Departamento de Matem\'atica, Universidad de la Frontera, Casilla 54-D, Temuco, Chile}
\email{ruben.hidalgo@ufrontera.cl, saul.quispe@ufrontera.cl}

\thanks{R. A. Hidalgo and S. Quispe are partially supported by projects Fondecyt 1150003 and 3140050 and Proyecto Anillo CONICYT PIA ACT 1415}

\maketitle

%%%%%%%%%%%%
\begin{abstract}
The known (explicit) examples of Riemann surfaces not definable over their field of moduli are not real whose field of moduli is a subfield of the reals. In this paper we provide explicit families of real Riemann surfaces which cannot be defined over the field of moduli.

\end{abstract}

%%%%%%%%%%%%%
%%%%%%%%%%%%%
\section{Introduction}\label{s1}
As a consequence of Riemann-Roch's theorem, each Riemann surface $\mathcal{S}$ of genus $g$ can be defined by an irreducible complex projective curve $\mathcal{C}$. A subfield $K$ of ${\mathbb C}$ is called a {\em field of definition} for $\mathcal{S}$ if it is possible to assume $\mathcal{C}$ to be defined by polynomials with coefficients in $K$. By results due to Koizumi \cite{Koi}, the intersection of all fields of definition of $\mathcal{S}$ is the {\em field of moduli} of $\mathcal{S}$ and there is a field of definition being a finite extension of the field of moduli. The surface $\mathcal{S}$ is called real if ${\mathbb R}$ is a field of definition of it; this is equivalent for $\mathcal{S}$ to admit an anticonformal automorphism of order two (as a consequence of Weil's descent theorem \cite{We}). Also, the field of moduli of $\mathcal{S}$ is a subfield of ${\mathbb R}$ if and only if it is isomorphic to its complex conjugate, equivalently, if it admits anticonformal automorphisms \cite{Earle, Shimura, Silhol}. Those surfaces of genus $g$ with real field of moduli corresponds to the real points of the moduli space ${\mathcal M}_{g}$. Riemann surfaces whose field of moduli is real but are not real are usually called {\it pseudo-real}.

It is well known that every Riemann surface of genus at most one can be defined over its field of moduli. So we assume, from now on, that $g \geq 2$. In this case, when ${\rm Aut}(\mathcal{S})$ is the trivial group or when $\mathcal{S}/{\rm Aut}(\mathcal{S})$ has genus zero and exactly $3$ cone points (that is, $\mathcal{S}$ is a quasiplatonic curve),
then $\mathcal{S}$ can be defined over its field of moduli (the first as a consequence of Weil's descent theorem \cite{We} and the second was proved by Wolfart \cite{Wolfart}).

Explicit examples of hyperelliptic Riemann surfaces which cannot be defined over the field of moduli were first provided by Earle \cite{Earle}, Shimura \cite{Shimura} and later by Huggins \cite{Hu1,Hu2}. In the non-hyperelliptic case, explicit examples were provided by the second author \cite{Hid}, Kontogeorgis \cite{Kontogeorgis} and the second and third author together with Artebani and Carvacho \cite{Aq, Achq}.
All of these examples are pseudo-real ones and, moreover, they can be defined over an imaginary extension of degree two of the field of moduli.

Because of the above examples, we were wondering if every real Riemann surface can be defined over its field of moduli.
In this paper we provide explicit examples of real Riemann surfaces (hyperelliptic and non-hyperelliptic) which are not definable over their field of moduli.

\subsection*{Notations and conventions}
By $L_D$ we mean the quadratic number field extension
$$\mathbb{Q}(\sqrt{D})=\{a+b\sqrt{D},\ a, b\in {\mathbb Q}\},$$
where $D>1$ is a square-free integer. Assume further that Pell's equation $a^2-Db^2=-1$ has a solution in $\mathbb{Z}^2$ (for instance, $D=2$). In particular, if $\sigma$ generates $\operatorname{Gal}(L_D/\mathbb{Q})$, then $L_D$ would contain infinitely many points $\eta=a+b\sqrt{D}$ whose norm $N_{L_D/\mathbb{Q}}(\eta)$ equals to $-1$, or equivalently $^{\sigma}\eta=-\eta^{-1}$.

We mainly are interested in the following two cases (i) $L/K$ is a finite Galois extension inside $\mathbb{C}$ and (ii) $K={\mathbb Q}$ and $L={\mathbb C}$.
The Galois group for $L/K$ is $\Gal(L/K)$, where its action will be denoted by left exponentiation. In particular, if $F\in L[X_0,\cdots,X_n]$ and $\sigma\in \Gal(L/K)$, then $^{\sigma}F$ denotes the  polynomial obtained by applying $\sigma$ to the coefficients of $F$.

The $n$-dimensional projective space over the complex field is $\mathbb{P}^{n}_{\mathbb{C}}$, and its automorphism group is $\operatorname{PGL}_{n+1}(\mathbb{C})$, the $(n+1)$-dimensional projective general linear group.
A projective linear transformation $A=(a_{i,j})$ of $\mathbb{P}^2_{\mathbb{C}}$ is often written as $[a_{1,1}X+a_{1,2}Y+a_{1,3}Z:a_{2,1}X+a_{2,2}Y+a_{2,3}Z:a_{3,1}X+a_{3,2}Y+a_{3,3}Z],$
where $\{X,Y,Z\}$ are the homogenous coordinates of $\mathbb{P}^2_{\mathbb{C}}$. The subgroup of all elements of $\operatorname{PGL}_3(\mathbb{C})$ of the shape
$$\left(
    \begin{array}{ccc}
      \ast & 0 & \ast \\
      0 & 1 & 0 \\
      \ast & 0 & \ast \\
    \end{array}
  \right)
$$
will be denoted by $\operatorname{GL}_{2,Y}(\mathbb{C})$.
We use $\zeta_n$ for a fixed primitive $n$-th root of unity in $\mathbb{C}$.

\subsection*{Acknowledgment}
The authors would like to thank Francesc Bars of Universitat Aut\`onoma de Barcelona, for his careful reading of an early version of this paper.

%%%%%%%%%%%%%%%%%%%%%%
%%%%%%%%%%%%%%%%%%%%%% a given and let $F/K$ be a Galois extension.
\section{Weil's criterion of descent}\label{s2}

Let $\mathcal{C}$ be a smooth projective algebraic curve defined over the field $L$, that is, $\mathcal{C}$ is defined as the zero locus of the homogeneous polynomials $F_1,\cdots,F_s\in L[X_0,\cdots,X_n]$.

For each $\sigma\in \Gal(L/K)$ the new polynomials $^{\sigma}F_1,\cdots,^{\sigma}F_s$ define a smooth projective curve $^{\sigma}\mathcal{C}$. In general, it may be that $^{\sigma}\mathcal{C}$ and $\mathcal{C}$ are not isomorphic curves.
\begin{defn}
The \emph{field of moduli} of $\mathcal{C}$ relative to the extension $L/K$, denoted by $M_{L/K}(\mathcal{C})$, is the fixed subfield of $L$ by the group
 $$U_{L/K}(\mathcal{C})=\{\sigma\in \Gal(L/K):\ \mathcal{C}\  \textrm{is isomorphic to }^{\sigma}\mathcal{C}\ \textrm{over } L\}.$$
\end{defn}

\begin{remark}
If $K={\mathbb Q}$ and $L={\mathbb C}$, then $M_{{\mathbb C}/{\mathbb Q}}(\mathcal{C})=M_{\overline{\mathbb Q}/{\mathbb Q}}(\mathcal{C})$ and this is the intersection of all fields of definition of $\mathcal{C}$ \cite{Koi}. This intersection property may fail, for instance, when (i) $K={\mathbb Q}$ and $L=\overline{\mathbb Q}$ or (ii) $K={\mathbb R}$ and $L={\mathbb C}$.
\end{remark}
%%%%%%%%%%%%%%%%
Let $L/K$ is a finite Galois extension and assume there exists an isomorphism $g:\mathcal{C}'\to \mathcal{C}$, defined over $L$, with $\mathcal{C}'$ defined over $K$. Then, (i) for each $\sigma\in \Gal(L/K)$, the rational map $f_{\sigma}:=g\circ\,(^{\sigma}g)^{-1}:\,^{\sigma}\mathcal{C}\to\mathcal{C}$ is an isomorphism defined over $L$, (ii) $f_{\sigma\tau}=f_{\sigma}\circ\, ^{\sigma}f_{\tau}$ holds for all $\sigma,\tau\in \Gal(L/K)$ and (iii) $f_{\sigma}\circ\,^{\sigma}g = g$.
The following theorem due to A. Weil shows that the above necessary conditions (i)-(iii) is also sufficient for the field $K$ to be a field of definition for $\mathcal{C}$.

\s
\noindent
\begin{theorem}[Weil's descent theorem \cite{We}]\label{Wa}
Let us assume that $L/K$ be a finite Galois extension and let
$\mathcal{C}$ be an irreducible projective algebraic curve, defined over $L$.
If for every $\sigma\in \Gal(L/K)$ there is an isomorphism $f_{\sigma}:\, ^{\sigma}\mathcal{C}\rightarrow \mathcal{C}$, defined over $L$, such that  the Weil's co-cycle
condition $f_{\sigma\tau} = f_{\sigma}\circ\,^{\sigma}f_{\tau}$ holds for all $\sigma, \tau\in \Gal(L/K)$,
then there exist an irreducible projective algebraic curve $\mathcal{C}'$ defined
over $K$ and an isomorphism $g: \mathcal{C}'\rightarrow \mathcal{C}$ defined over $L$ such that $f_{\sigma}\circ\,^{\sigma}g = g$. We say that the collection $\{f_{\sigma}\}_{\sigma \in \Gal(F/K)}$ is a Weil's datum for $\mathcal{C}$ with respect to the Galois extension $L/K$.
\end{theorem}

\s
\noindent
\begin{remark}
The above result still valid for the case $K={\mathbb Q}$ and $L={\mathbb C}$ if we assume $\mathcal{C}$ to be of genus $g \geq 2$ (so it has a finite group of automorphisms).
\end{remark}

\s
%%%%%%%%%%%%%%%%%%
\subsection{An application}
A constructive proof of Theorem \ref{Wa} can be found in \cite{HR}.  The constructive proof in fact asserts that if every $f_{\sigma}$ is defined over a subfield $N$, $K < N <L$, then $g$ is also defined over $N$. This observation will be important for our construction.

\s
\noindent
\begin{corollary}\label{coro1}
Let $L/K$ be a Galois extension of degree $2$ and let $\sigma$ be  the generator of $\Gal(L/K)\cong {\mathbb Z}/2{\mathbb Z}$.
Let $\mathcal{C}$ be an irreducible projective algebraic curve of genus $g \geq 2$, defined over $L$. Assume that every automorphism of $\mathcal{C}$ is also defined over $L$ and that there is an isomorphism $h:\,^{\sigma}\mathcal{C} \to \mathcal{C}$ also defined over $L$. Then, if there is no a Weil's descent datum for $\mathcal{C}$, with respect to the Galois extension $L/K$, then $\mathcal{C}$ cannot be defined over $K$.
\end{corollary}
\begin{proof}
Let us assume $\mathcal{C}$ is definable over $K$. It follows that there is an isomorphism $t:\mathcal{C}' \to \mathcal{C}$, where $\mathcal{C}'$ is defined over $K$. As the group of automorphisms of $\mathcal{C}$ is finite, the isomorphism $t$ can be assumed to be defined over a finite extension $M$ of $L$; which can be assume to be a Galois extension of $K$.

If $\eta \in {\rm Gal}(M/K)$, then the isomorphism $f_{\eta}:=t \circ\, (^{\eta}t)^{-1}:\,^{\eta}\mathcal{C} \to \mathcal{C}$ must be of the form $a_{\eta} \circ h $, where $a_{\eta}$ is an automorphism of $\mathcal{C}$; it follows that $f_{\eta}$ is defined over $L$. Moreover, it can be seen that the collection $\{f_{\eta}: \eta \in {\rm Gal}(M/K)\}$ is a Weil's datum for $\mathcal{C}$ with respect to the finite Galois extension $M/K$. Using this collection of isomorphisms in Weil's descent theorem and its constructive proof in \cite{HR}, we obtain an isomorphism $g:\mathcal{C}'' \to \mathcal{C}$, where $\mathcal{C}''$ is defined over $K$ and $g$ is defined over $L$.

If ${\rm Gal}(L/K)=\langle \sigma \rangle$, then $\{t_{\sigma}=g \circ\, (^{\sigma}g)^{-1} , t_{e}=I\}$ defined a Weil's datum for $\mathcal{C}$ with respect to the Galois extension $L/K$, a contradiction.
\end{proof}

%\s
\noindent
\begin{remark}
The above corollary works for $\mathcal{C}$ any algebraic variety with a finite group of automorphisms.
\end{remark}

\s

%%%%%%%%%%%%%%%%%%%%%%%
%%%%%%%%%%%%%%%%%%%%%%%%> Genus(C);
\section{Hyperelliptic real curves not definable over their field of moduli}\label{s3}
In this section we provide examples of hyperelliptic curves, over a real quadratic field extension, not definable over their field of moduli.

Fix an even integer $g\geq2$, and choose $\eta_i$, for $1\leq i\leq g-1$, in $$L_D\setminus\big(\mathbb{Q}\cup\{\pm\sqrt{D},\frac{\pm1}{\sqrt{D}}\}\big)$$ with the property that $N_{L_D/\mathbb{Q}}(\eta_i)=-1$, and such that $\eta_i\neq\frac{\pm1}{\eta_j}$ for any $i,j$.
Now, for $t\in\mathbb{Q}\setminus\{0,\pm1\}$, consider the family of Riemann surfaces $\mathcal{C}_{t,g}$ defined in $\mathbb{P}^2_{\mathbb{C}}$ by an equation of the form
$$Y^2Z^{2g}-(X+tZ)(X+t^{-1}Z)(X+\sqrt{D}Z)(X-\frac{1}{\sqrt{D}}Z)\prod_{i=1}^{g-1}(X^2-\eta_i^2Z^2).$$
It is clear that $\mathcal{C}_{t,g}$ is hyperelliptic of genus $g$, with hyperelliptic involution $\iota:(X:Y:Z)\mapsto(X:-Y:Z).$
\begin{theorem}\label{familysmoothauto2}
Following the above notations, there exist infinitely many $t\in\mathbb{Q}\setminus\{0,\pm1\}$ such that $\mathcal{C}_{t,g}$ has automorphism group $$\operatorname{Aut}(\mathcal{C}_{t,g})=\langle\iota\rangle\simeq\mathbb{Z}/2\mathbb{Z}.$$
Moreover, the field of moduli is $\mathbb{Q}$, but it is not a field of definition.
\end{theorem}
\begin{proof}
 The $g$-fold branched cover $\pi:\mathcal{C}_{t,g}\rightarrow\widehat{\mathbb{C}}$ has branch values, given by the real points $$-t,-t^{-1},-\sqrt{D},\frac{1}{\sqrt{D}},\pm\eta_1,...,\pm\eta_{g-1}.$$
Hence, any automorphism of $\mathcal{C}_{t,g}$ in the reduced automorphism group $\overline{\operatorname{Aut}(\mathcal{C}_{t,g})}:=\operatorname{Aut}(\mathcal{C}_{t,g})/\langle\iota\rangle$ induces a M\"{o}bius transformation
$$T:x\mapsto\frac{ax+b}{cx+d},$$
leaving invariant the branch locus of $\pi$. One easily checks that our restrictions on $\eta_i's$ imposes $T:t\mapsto t$ or $t^{-1}$ and $\sqrt{D}\mapsto \sqrt{D}$ or $\frac{-1}{\sqrt{D}}$. In particular, there are only finitely
many possibilities for the orbit of $t$ under $T$. Each possibility will provide
a polynomial equation on $t$ in terms of $\sqrt{D}$ and $\eta_i's$. After excluding these finitely many possibilities, we conclude that $\operatorname{Aut}(\mathcal{C}_{t,g})=\langle\iota\rangle$, for infinitely many values of $t\in\mathbb{Q}\setminus\{0,\pm1\}$.

Now, the map
$$f_{\sigma}:\big(X:Y:Z\big)\mapsto\big(X^gZ:\big(\prod_{i=1}^{g-1}\eta_i\big)YZ^g:X^{g+1}\big)$$
defines an isomorphism between $^{\sigma}\mathcal{C}_{t,g}{\stackrel{f_{\sigma}}{\longrightarrow}}\mathcal{C}_{t,g}$ (recall that $g\geq2$ is even). Therefore, the field of moduli is $\mathbb{Q}$. On the other hand, any other isomorphism $f_{\sigma}':\,^{\sigma}\mathcal{C}_{t,g}\rightarrow \mathcal{C}_{t,g}$ equals to $f_{\sigma}$ or $\iota \circ f_{\sigma}$. Since $f_{\sigma}$ and $\iota$ commutes, and $f_{\sigma}\circ\,^{\sigma}f_{\sigma}\neq1$, then we also get $f_{\sigma}'\circ\,^{\sigma}f_{\sigma}'\neq1$. In particular, the Weil's cocylce criterion of decent is not verified, and $\mathbb{Q}$ is not a field of definition for $\mathcal{C}_{t,g}$.
\end{proof}
%%%%%%%%%%%%%%%%%%%%%%%%%%%%%%%%%%%%%%%

\section{On automorphism groups of smooth plane curves}
By a smooth plane curve over $\mathbb{C}$ of degree $d\geq4$, we mean a smooth curve $C$ over $\mathbb{C}$, which is $\mathbb{C}$-isomorphic to a non-singular plane model $F_C(X,Y,Z)=0$ in $\mathbb{P}^2_{\mathbb{C}}$, where $F_C(X,Y,Z)$ is a homogenous polynomial of degree $d$ with complex coefficients.

Using elementary algebraic geometry (Riemann-H\"{u}rwitz formula and B\'{e}zout theorem) one shows the following:
\begin{lemma}\label{nonhypergenus}
Let $C\hookrightarrow\mathbb{P}^2_{\mathbb{C}}$ be a smooth plane curve of degree $d\geq4$. Then, it is non-hyperelliptic of genus $g=(d-1)(d-2)/2$.
\end{lemma}
\begin{proof}
Let $H\subset\mathbb{P}^2_{\mathbb{C}}$ be a hyperplane section of $C$. In particular, the canonical divisor $K$ of $C$ is $\sim(d-3)(H\cap C)$. By B\'{e}zout theorem $H\cap C$ has degree exactly $d$. Therefore, Riemann-H\"{u}rwitz reads
$$2g-2=\operatorname{deg}(K)=(d-3)d,$$
that is $g=(d-1)(d-2)/2$.

Next, if $f(x,y)=0$ is the affine equation of a smooth plane curve $C$ of degree $d\geq4$, then
$$\big\{\frac{x^ry^s}{f_y}\,|\,0\leq r+s\leq d-3\big\}$$
is a basis of the space of regular differentials on $C$. Therefore, the canonical map $C\rightarrow\mathbb{P}_{\mathbb{C}}^{g-1}$ can be seen as the map
$$(x:y:1)\mapsto(x^ry^s\,|\,0\leq r+s\leq d-3).$$
In particular, when $d=4$, this map is exactly the identity map, and hence is an embedding. Thus a smooth plane curve $C$ of degree $d=4$ over $\mathbb{C}$ is non-hyperelliptic. Now, assume that $d\geq5$ and $C$ is hyperelliptic. Hence, it has a hyperelliptic involution $\iota$ of order $2$, which fixes exactly $2g+2=(d-4)(d+1)$ points on $C$. Thinking about $\iota$, up to $\operatorname{PGL}_{3}(\mathbb{C})$-conjugation, as the automorphism $[X:Y:-Z]$, gives at most $d$ fixed points on $C$, since $\iota$ leaves invariant in $\mathbb{P}^2_{\mathbb{C}}$, the line $Z=0$, the point $(0:0:1)\notin C$ and no other points. % Z^2=f(X,Y) so (0:0:1)\notinC
That is, $(d-4)(d+1)\leq d$, a contradiction!. That is, a smooth plane curve $C$ over $\mathbb{C}$ of degree $d\geq5$ is always non-hyperelliptic.
\end{proof}
\begin{defn}
For a non-zero monomial $cX^iY^jZ^k$
with $c\in\mathbb{C}\setminus\{0\}$, its exponent is defined to be $max\{i,j,k\}$. For a homogenous polynomial $F(X,Y,Z)$, the core of it is
defined to be the sum of all terms of $F$ with the greatest
exponent. Now, let $C_0$ be a smooth plane curve, a pair $(C,H)$ with
$H\leq \operatorname{Aut}(C)$ is said to be a descendant of $C_0$ if $C$ is defined
by a homogenous polynomial whose core is a defining polynomial of
$C_0$ and $H$ acts on $C_0$ under a suitable change of the
coordinates system, i.e. $H$ is conjugate to a subgroup of
$\operatorname{Aut}(C_0)$.
\end{defn}
By \cite[\S 1-10]{Mit} and the proof of Theorem 2.1 in \cite{Ha}, we conclude:
\begin{theorem}[Mitchell \cite{{Mit}}, Harui \cite{Ha}]\label{Harui,Mitchell}
	Let $G$ be a subgroup of automorphisms of a smooth plane curve $C$ of degree $d\geq4$ defined over $\mathbb{C}$. Then, one of the following holds:
	\begin{enumerate}[(i)]
		\item $G$ fixes a line in $\mathbb{P}^2_{\mathbb{C}}$ and a point off this line.
		\item $G$ fixes a triangle $\Delta\subset\mathbb{P}^2_{\mathbb{C}}$, i.e. a set of three non-concurrent lines, and neither line nor a point is leaved invariant. In this case, $(C,G)$ is a descendant of the the Fermat curve $F_d:\,X^d+Y^d+Z^d=0$ or the Klein curve $K_d:\,XY^{d-1}+YZ^{d-1}+ZX^{d-1}=0$.
		\item $G$ is $\operatorname{PGL}_3(\mathbb{C})$-conjugate to a finite primitive subgroup namely, the Klein group
		$\operatorname{PSL}(2,7)$, the icosahedral group $\operatorname{A}_5$, the alternating group
		$\operatorname{A}_6$, the Hessian group $\operatorname{Hess}_{*}$ with $*\in\{36,72,216\}$.
		\end{enumerate}
\end{theorem}
\begin{defn}\label{homologydefn}
By an homology of period $n\in\mathbb{Z}_{\geq1}$, we mean a projective linear transformation of the plane $\mathbb{P}^2_{\mathbb{C}}$, which acts up to $\operatorname{PGL}_3(\mathbb{C})$-conjugation, as $$(X:Y:Z)\mapsto (\zeta_{n}X:Y:Z).$$
Such a transformation fixes pointwise a line (its axis) and a point off this line (its center).
\end{defn}
\begin{theorem}[Mitchell \cite{Mit}]\label{homologies}
Let $G$ be a finite group of $\operatorname{PGL}_3(\mathbb{C})$. If $G$ contains an homology of period $n\geq4$, then it fixes a point, a line or a triangle. Moreover, the Hessian group $\operatorname{Hess}_{216}$ is the only finite subgroup of $\operatorname{PGL}_3(\mathbb{C})$ that contains homologies of period $n=3$, and does not leave invariant a point, a line or a triangle.
\end{theorem}

%%%%%%%%%%%%%%%%%%%%%%%%%%%%%%%%%%%%%%%%

\section{Non-hyperelliptic real curves not definable over their field of moduli}\label{s3}
For any arbitrary integer $d=4m>0$, fix $\{\eta_1,\eta_2,...,\eta_m\}\subset L_D$, satisfying the following conditions:
\begin{enumerate}[(i)]
  \item (Non-singularity) For $t\in\mathbb{Q}\setminus\{0,\pm1\}$, the form $$f_t(X,Z):=\prod_{i=1}^m\,(X^2-\eta_i^{-2}Z^2)(X+t\eta_i^3Z)(X-t^{-1}\eta_i^3Z)$$ has no repeated zeros.
  \item (Automorphism group) The form $$g(X,Z):=\prod_{i=1}^m\,(X-\frac{\eta_i}{\sqrt{D}}Z)(X-\eta_i\sqrt{D}Z)$$ is not invariant under any automorphism of $\mathbb{P}^1_{\mathbb{C}}$ of the shape
      $$
      \psi_a:(X:Z)\to(Z:aX)\,\,\text{or}\,\,\psi_{a,b}:(X:Z)\to(X+aZ:bX-Z),
      $$
 with $a,b\in\mathbb{C}$.
\item (Weil's datum) The norms $N_{L_D/\mathbb{Q}}(\eta_i)=-1$, for all $i$.
\end{enumerate}

\begin{theorem}\label{familysmoothauto}
Following the above notations, consider the family of Riemann surfaces $\mathcal{S}_{t,d}$, with $d\geq 4$, given in $\mathbb{P}^2_{\mathbb{C}}$ by the equation
\begin{equation}\label{nonhypequation}
F(X,Y,Z):=Y^d+Y^{d/2}g(X,Z)+f_t(X,Z)=0.
\end{equation}
Then, there exist infinitely many $t\in\mathbb{Q}\setminus\{0,\pm1\}$ such that $\mathcal{S}_{t,d}$ is non-hyperelliptic of genus $g=(d-1)(d-2)/2$. The full automorphism group $\operatorname{Aut}(\mathcal{S}_{t,d})$ is $\operatorname{PGL}_3(\mathbb{C})$-conjugate to the cyclic group $\mathbb{Z}/(d/2)\mathbb{Z}$, generated by $[X:\zeta_{d/2}Y:Z]$. Moreover, the field of moduli of $\mathcal{S}_{t,d}$ is $\mathbb{Q}$, but is not a field of definition.
\end{theorem}

\begin{proof}
First, we show that the Riemann surface $\mathcal{S}_{t,d}$ is non-hyperelliptic of genus $g=(d-1)(d-2)/2$, for any $t\in\mathbb{Q}\setminus\{0,\pm1\}$. By Lemma \ref{nonhypergenus}, it suffices to see that the equation (\ref{nonhypequation}) has no singular points in $\mathbb{P}^2_{\mathbb{C}}$. Since $F(X,Y,0)=Y^d+(XY)^{\frac{d}{2}}-X^d=0$ has no repeated zeros, the common zeros of $F(X,Y,0)$ and $\frac{\partial F}{\partial X}(X,Y,0)$ do not exist. Furthermore, $\frac{\partial F}{\partial X}(X,Y,1)=Y^{\frac{d}{2}}g'(X,1)-f_t'(X,1)$ and $\frac{\partial F}{\partial Y}(X,Y,1)=\frac{d}{2}Y^{\frac{d}{2}-1}(2Y^{\frac{d}{2}}+g(X,1))$. But $f_t(X,Z)$ is square free, then $(X:0:1)$ gives no singularities on $F(X,Y,Z)=0$. Also, if we substitute into $F(X,Y,1)=0$, we get that $F(X,Y,Z)=0$ is singular only if $g(X,1)^2=-4f_t(X,1)$, % divide by z^{} and substitute from the first equation into the second.
which is not possible because $f_t(X,Z)$ is again square-free. So, the equation produces a non-singular plane model of $\mathcal{S}_{t,d}$ over $\mathbb{C}$ of degree $d\geq4$. In particular, $\mathcal{S}_{t,d}$ is non-hyperelliptic of genus $g=\frac{1}{2}(d-1)(d-2)$, using Lemma \ref{nonhypergenus}.

Next, assume that the claim on $\operatorname{Aut}(\mathcal{S}_{t,d})$ is true. Consider the Galois extension $\mathbb{L}:=L_D(\zeta_{d/2})$ of $\mathbb{Q}$, where $\mathcal{S}_{t,d}$ and all of its automorphisms are defined. More precisely, $\mathcal{S}_{t,d}$ is defined over $L_D=\mathbb{Q}(\sqrt{D})$, whereas $\operatorname{Aut}(\mathcal{S}_{t,d})$ is defined over the cyclotomic extension $\mathbb{Q}(\zeta_{d/2})$. Let $\tau$ be a generator of $\operatorname{Gal}(\mathbb{Q}(\zeta_{d/2})/\mathbb{Q})$, hence the group $\operatorname{Gal}(\mathbb{L}/\mathbb{Q})=\langle\sigma,\tau\rangle$. Moreover, one easily checks that $\mathcal{S}_{t,d}$ is isomorphic to its conjugates $^{\sigma}\mathcal{S}_{t,d}$ and $^{\tau}\mathcal{S}_{t,d}=\mathcal{S}_{t,d}$ via $f_{\sigma}=[Z:\alpha Y:X]$ and $f_{\tau}:=1$, where $\alpha^{d/2}=(\prod_{i=1}^m\,\eta_i)^2$. Hence, $\mathbb{Q}$ is the field of moduli for $\mathcal{S}_{t,d}$ relative to $\mathbb{L}/\mathbb{Q}$. However, it is not a field of definition for $\mathcal{S}_{t,d}$. To see this, let $f:\,^{\sigma}\mathcal{S}_{t,d}\rightarrow \mathcal{S}_{t,d}$ be any other isomorphism. Then $f_{\sigma}\circ f^{-1}\in\operatorname{Aut}(\mathcal{S}_{t,d})$, and so $f=f_{\sigma}\circ[X:\zeta_{d/2}Y:Z]^{m}$, for some integer $0\leq m<d/2$. Any such $f$ does not satisfy Weil's condition of descent (see section \S \ref{s2}), since $f\circ\,^{\sigma}f=[Z:\zeta_{d/2}^m\alpha Y:X]\neq f_{e}$, where $e$ is the trivial automorphism of $\mathbb{L}$. Therefore, $\mathbb{Q}$ is not a field of definition for $\mathcal{S}_{t,d}$.

Lastly, it remains to prove our claim on $\operatorname{Aut}(\mathcal{S}_{t,d})$. We consider the following two cases:

\noindent (Case $d\not=4$).
We use quite similar techniques as the ones in \cite{BaBa1, BaBa3, BaBa2}. It is clear that $\psi:=[X:\zeta_{d/2}Y:Z]\in\operatorname{Aut}(\mathcal{S}_{t,d})$ is an homology of order $d/2\geq4$ (Definition \ref{homologydefn}). Therefore, $\operatorname{Aut}(\mathcal{S}_{t,d})$ fixes a point, a line or a triangle, by Theorem \ref{homologies}. In particular, it is not conjugate to any of the finite primitive group mentioned in Theorem \ref{Harui,Mitchell}-(iii). Now, we treat each of the following subcases:
\begin{enumerate}[(i)]
\item A line $L\subset\mathbb{P}^{2}_{\mathbb{C}}$ and a point $P\notin L$ are leaved invariant: By Theorem 2.1 in \cite{Ha}, we can think about $\operatorname{Aut}(\mathcal{S}_{t,d})$ in a short exact sequence
\scriptsize
$$
\xymatrix
{
1\ar[r]  & \mathbb{C}^*\ar[r]                    & \operatorname{GL}_{2,Y}(\mathbb{C})\ar[r]^{\varrho}& \operatorname{PGL}_2(\mathbb{C})\ar[r]& 1         \\
         &                              &                                        &                              &\\
1\ar[r]  & \langle\psi\rangle\ar[r]\ar@{^{(}->}[uu] & \operatorname{Aut}(\mathcal{S}_{t,d})\ar[r]\ar@{^{(}->}[uu] & G\ar[r]\ar@{^{(}->}[uu]& 1
}
$$
\normalsize
where $G$ is conjugate to a cyclic group $\mathbb{Z}/m\mathbb{Z}$ of
order $m\leq d-1$, a Dihedral group $\operatorname{D}_{2m}$ of order $2m$
with $m|(d-2)$, one of the alternating groups $\operatorname{A}_4$, $\operatorname{A}_5$, or to
the symmetry group $\operatorname{S}_4$. Any such $G$, which is not cyclic, contains an element of order $2$. Let $\psi'\in\operatorname{Aut}(\mathcal{S}_{t,d})$ such that $\varrho(\psi')$ has order $2$. Then, $\varrho(\psi')$ has the shape $\psi_a$ or $\psi_{a,b}$ for some $a,b\in\mathbb{C}\setminus\{0\}$, which is absurd by our assumptions on $g(X,Z)$. Consequently, $G=\varrho(\operatorname{Aut}(\mathcal{S}_{t,d}))$ is cyclic, generated by the image of a specific $\psi_G\in \operatorname{GL}_{2,Y}(\mathbb{C})$. In the worst case, this would lead to a polynomial expression of $t$ in terms of the $\eta_i's$, which associates only finitely many values for $t$ with $|G|>1$. So we still have infinitely many $t\in\mathbb{Q}\setminus\{0,\pm1\}$ such that $f_t(X,Z)$ not $\langle\varrho(\psi_G)\rangle$-invariant. In particular, $|G|=1$ and $\operatorname{Aut}(\mathcal{S}_{t,d})$ is $\operatorname{PGL}_3(\mathbb{C})$-conjugate to $\langle[X:\zeta_{d/2}Y:Z]\rangle$.
\item A triangle $\Delta$ is fixed by $\operatorname{Aut}(\mathcal{S}_{t,d})$ and neither a line nor a point is leaved invariant: It follows by the proof of Theorem 2.1 in \cite{Ha} that $(\mathcal{S}_{t,d},\operatorname{Aut}(\mathcal{S}_{t,d}))$ should be a descendant of the Fermat curve $F_d$ or the Klein curve $K_d$ as in Theorem \ref{Harui,Mitchell}. Note that $d/2$ does not divide $|\operatorname{Aut}(K_d)|=3(d^2-3d+3)$,
    cf. \cite[Propositions 3.5]{Ha}. Therefore, $(\mathcal{S}_{t,d},\operatorname{Aut}(\mathcal{S}_{t,d}))$ is not a descendant of $K_d$. Hence, $\exists \phi\in\operatorname{PGL}_3(\mathbb{C})$ where $H:=\phi^{-1}\operatorname{Aut}(\mathcal{S}_{t,d})\phi\leq\operatorname{Aut}(F_d)$. It is also well known (cf. \cite[Proposition 3.3]{Ha}) that $\operatorname{Aut}(F_d)$ is a semidirect product of $\operatorname{S}_3=\langle T:=[Y:Z:X],R:=[X:Z:Y]\rangle$ acting on $(\mathbb{Z}/d\mathbb{Z})^2=\langle[\zeta_dX:Y:Z],[X:\zeta_dY:Z]\rangle.$ Thus any element of $\phi^{-1}\operatorname{Aut(\mathcal{S}_{t,d})}\phi$ has the shape $DR^iT^j$, for some $0\leq i\leq1$ and $0\leq j\leq2$ and $D$ is of diagonal shape in $\operatorname{PGL}_3(\mathbb{C})$. It is straightforward to check that any $DT^j$ and $DRT^j$ with $j\neq0$ has order $3<d/2$. Thus $\phi^{-1}\psi\phi$ has also a diagonal shape, and then we may take $\phi$ in the normalizer of $\langle\psi\rangle$, up to a change of variables in $\operatorname{Aut}(F_d)$.  In this case, we can think about $\operatorname{Aut}(\mathcal{S}_{t,d})$ in the commutative diagram
    \scriptsize
    $$
    \xymatrix
    {
     1\ar[r]  & (\mathbb{Z}/d\mathbb{Z})^2\ar[r]                    & \operatorname{Aut}(F_d)\ar[r]^{\varrho}& \operatorname{S}_3\ar[r]& 1         \\
         &                              &                                        &                              &\\
    1\ar[r]  & \operatorname{Ker}(\varrho|_{H})=\langle\psi\rangle\ar[r]\ar@{^{(}->}[uu] & H\ar[r]\ar@{^{(}->}[uu] & G:=\operatorname{Im}(\varrho|_{H})\ar[r]\ar@{^{(}->}[uu]& 1
    }
    $$
\normalsize
% remark because P\in GL_2,z, we only have the powers of Z^d Z^{d/2} even in the transofrmaed form. This implies \operatorname{Ker}(\varrho|_{\operatorname{Aut}}(C))
The variable $Y$ in the transformed defining equation via $\phi$ appears exactly as the original equation in the statement. Hence, $G$ is at most cyclic of order $2$, since otherwise $H$ must have an element of the shape $[\zeta_d^mY:\zeta_d^nZ:X]$ or $[\zeta_d^mZ:\zeta_d^nX:Y]$, for some integers $m,n$, which is not possible. For the same reason, $G$ is then generated by a certain $\varrho([\zeta_d^mY:\zeta_d^nX:Z])$, and as before, it only requires to exclude finitely many $t$ such that $f_t(\phi(X,Z))$ is not $\langle\varrho([\zeta_d^mY:\zeta_d^nX:Z])\rangle$-invariant, where $\phi$ is the restriction of $\phi$ on $\mathbb{C}[X,Z]$.
\end{enumerate}

\noindent (Case $d=4$). We use the method applied in \cite[Lemma 4.2]{Aq}, with $D=2$, $\eta=1+\sqrt{2}$ and $t=3$, in order to show that $\operatorname{Aut}(\mathcal{S}_{t,d=4})=\langle[X:-Y:Z]\rangle$.  In this case, the defining equation of $\mathcal{S}_{t,d=4}$ in $\mathbb{P}^2_{\mathbb{C}}$ reduces to
$$Y^{4}+Y^{2}(X-\dfrac{\eta}{\sqrt{D}}Z)(X-\eta\sqrt{D}Z)+(X^{2}-\eta^{-2}Z^{2})(X+t\eta^{3}Z)(X-t^{-1}\eta^{3}Z).$$
%
%This shows the claim on $\operatorname{Aut}(\mathcal{S}_{t,d})$, and we are done.
%\end{enumerate}
%\end{proof}
%Theorem \ref{familysmoothauto} is also true when $d=4$. The only difference is how we prove the claim on $\operatorname{Aut}(\mathcal{S}_{t,d=4})$. In this case, the defining equation in $\mathbb{P}^2_{\mathbb{C}}$ reduces to
%$$Y^{4}+Y^{2}(X-\dfrac{\eta}{\sqrt{D}}Z)(X-\eta\sqrt{D}Z)+(X^{2}-\eta^{-2}Z^{2})(X+t\eta^{3}Z)(X-t^{-1}\eta^{3}Z).$$
%We use the method applied in \cite[Lemma 4.2]{Aq}, with $D=2$, $\eta=1+\sqrt{2}$ and $t=3$, in order to show that $\operatorname{Aut}(\mathcal{S}_{t,d=4})=\langle[X:-Y:Z]\rangle$.
%\begin{prop}
%For infinitely many values of $t\in {\mathbb Q}\setminus\{0,\pm 1\}$, we have that
%$\operatorname{Aut}(\mathcal{S}_{t,d=4})=\langle\psi\rangle$ where $\psi:=[X:-Y:Z]$.
%\end{prop}
%\begin{proof}
Then, if $\Aut(\mathcal{S}_{t,d=4})\leq{\rm PGL}_{3}({\mathbb C})$ contains properly $\langle\psi\rangle$, then by \cite[pag. 26]{ba}, $\Aut(\mathcal{S}_{t,d=4})$ contains a subgroup isomorphic to either  ${\mathbb Z}/{2}\mathbb{Z} \times{\mathbb Z}/{2}\mathbb{Z}$, ${\mathbb Z}/{6}\mathbb{Z}$ or $\operatorname{S}_{3}$. We will now exclude each of these cases.

The first case can be excluded because an explicit computation shows that there is no involution, except $\psi$, which preserves the four fixed points of $\psi$ on $\mathcal{S}_{t,d=4}$.

The second case can be excluded because if  $\Aut(\mathcal{S}_{t,d=4})$ contains a cyclic subgroup of order $6$ generated by $\phi$ with $\psi=\phi^{3}$, then the automorphism $\vartheta=\phi^{2}$ induces an order three automorphism $\bar{\vartheta}$ on the elliptic curve $E:=\mathcal{S}_{t,d=4}/\langle\psi\rangle$ having fixed points. This is a contradiction, since the curve $E$ (whose equation can be obtained replacing $Y^{2}$ with $Y$ in the equation of $\mathcal{S}_{t,d=4}$) has $j$-invariant distinct from zero.

Finally, suppose that $\Aut(\mathcal{S}_{t,d=4})$ contains a subgroup $\langle\psi, \phi\rangle$ isomorphic to $\operatorname{S}_{3}$. Here we will apply a method suggested by F. Bars \cite{ba}. By \cite[Theorem 29]{ba}, up to a change of coordinates the equation of $\mathcal{S}_{t,d=4}$ takes the following form:
$$(u^{3}+v^{3})w+u^{2}v^{2}+ruvw^{2}+sw^{4}=0,$$ and the generators of $\operatorname{S}_{3}$ with respect to the coordinates $(u, v, w)$ are
 \[ \theta:=\left( \begin{array}{ccc}
0 & 1 & 0 \\
1 & 0 & 0 \\
0 & 0 & 1 \end{array} \right), \quad \vartheta:=\left( \begin{array}{ccc}
\zeta_3 & 0 & 0 \\
0 & \zeta_3^2 & 0 \\
0 & 0 & 1 \end{array} \right).\]
Thus there exists $A\in \PGL_3(\mathbb{C})$ such that $A\theta A^{-1}=\psi, A\vartheta A^{-1}=\phi.$
 The first condition implies that $A$ is an invertible matrix of the following form
 \[ A=\left( \begin{array}{ccc}
a & a & c \\
e & -e & 0 \\
b & b & l \end{array} \right).
\]

Note that $\mathcal{S}_{t,d=4}$ has exactly four bitangents $X=s_jZ,\ j=1,2,3,4$ invariant under the action of the involution $\psi$,
where $s_j$  are the zeros of
$$P(X)=(X-\frac{1+\sqrt{2}}{\sqrt{2}})^2(X+\frac{\sqrt{2}}{1-\sqrt{2}})^{2}-$$ $$-4(X+(1-\sqrt{2}))(X+\frac{1}{1+\sqrt{2}})(X-3(7+5\sqrt{2}))(X+\frac{1}{3(7-5\sqrt{2})}).
$$

Let $b_{j1}=(s_j,q_j,1), b_{j2}=(s_j,-q_j,1)$ be the two tangency points of the line $X=s_jZ$.
On the other hand, observe that the line $w=0$ is invariant for $\theta$ and it is bitangent to $\mathcal{S}_{t,d=4}$ at $p_1=(1:0:0), p_2=(0:1:0)$.
Thus, for some $j$ we must have $\{Ap_1,Ap_2\}=\{b_{j1}, b_{j2}\}$,
 from which we get $a=s_j b$ and $e=\pm q_j b$. By means of these remarks and using MAGMA \cite{Magma} (see also \cite{HQArxiv}), we may see that  $\phi=A\vartheta A^{-1}$ is not an automorphism of $\mathcal{S}_{t,d=4}$.

This shows the claim on $\operatorname{Aut}(\mathcal{S}_{t,d})$, and we are done.
\end{proof}

\begin{remark}
The choice of $\eta_i's$ in Theorem \ref{familysmoothauto} such that the zero set of $g(X,Y)$ is not preserved under the action of any $\psi_a$ or $\psi_{a,b}$, is not restrictive as it looks. It only imposes finitely many algebraic conditions on the $\eta_i's$: For instance, any $\psi_{a}$ acts as a product of pairwise disjoint $2$-cycles on the set $\{(\eta_i\sqrt{D}:1),(\frac{\eta_i}{\sqrt{D}}:1)\}_i$ since it has order $2$ in $\operatorname{PGL}_{2}(\mathbb{C})$. Hence, if $\psi_a:(\eta_n\sqrt{D}:1)\leftrightarrow(\eta_m\sqrt{D}:1)$ (resp. $(\frac{\eta_m}{\sqrt{D}}:1)$) for some $n,m$, then $a=\frac{1}{\eta_n\eta_mD}$ (resp. $\frac{1}{\eta_n\eta_m}$). Therefore, it suffices to choose the $\eta_i's$ such that $\{\eta_i\}_{i}\neq\{\frac{\eta_n\eta_m}{\eta_i}\},\{\frac{\eta_n\eta_m}{\eta_iD}\}_i$ for any $m,n$. In this case, $g(X,Y)$ is not $\psi_a$-invariant for any $a\in\mathbb{C}$.

The action of an $\psi_{a,b}$ can be treated in the same way. However, it is a bit more tedious, and we skip it for simplicity.
\end{remark}

%%%%%%%%%%%%%%%%%

\bibliographystyle{amsplain}

\end{document}